\documentclass[12pt,twoside]{article}
\usepackage{amssymb,amsmath,amstext,amsthm,amsfonts}
\usepackage[ansinew]{inputenc} 
\usepackage{graphicx}

\usepackage{amsfonts}
\usepackage{graphicx}
\usepackage{amsmath}
\usepackage{amssymb}
\usepackage{color}
\usepackage[mathscr]{eucal}

\newcommand{\R}{\mathbb{R}}

\newcommand{\Mod}[1]{\left\vert{#1}\right\vert}
\newcommand{\nrm}[1]{\left\|#1\right\|}

\newcommand{\Uscr}{\mathscr{U}}

\newtheorem{defi}{Definition}
\newtheorem{lema}{Lemma}
\newtheorem{rmk}{Remark}
\newtheorem{prop}{Proposition}
\newtheorem{teor}{Theorem}
\newtheorem{coro}{Corollary}

\begin{document}

\title{r-Regularity
}

\author{Pedro Duarte         \and
        Maria Joana Torres
}



\maketitle

\begin{abstract}
We provide a characterization of $r$-regular sets in terms of the Lipschitz re\-gu\-la\-rity of normal vector fields to the boundary.

\end{abstract}

\maketitle

\section{Introduction}
\label{sec:0}
The fundamental task of digital image processing is to recognize the properties of
real objects given their digital images, i.e., discrete data generated by some image acquisition system.
An important question for devising reliable image analysis algorithms is:
{\em which shapes can be digitized without changes in the fundamental geometric or topological properties?}
Most of the known answers to this question consider, as suitable shapes, the class
of $r$-{\em regular sets} (see \cite{KS,LCG,P,Se,S,SK,SLS}).  
Furthermore, $r$-regular sets have been applied in the context of surface reconstruction and image segmentation \cite{MKS,S1,ST},
which shows that the topic of this paper ($r$-regularity) is very important in the context of many applications. These applications also include the
authors' motivation which comes from the study of
smooth non-deterministic dynamical systems, i.e., the dynamics of `smooth' point-set maps on
a compact manifold, where $r$-regular sets can appear as dynamically invariant sets (see~\cite{DT}).

In~\cite{LCG},
conditions were derived relating properties of regular sets to the grid size of the sampling device which guarantee that a regular object and its digital image are
topologically equivalent. To obtain the topological equivalence the authors
used the fact that a  regular set is always bounded by a codimension one manifold.
This property was conjectured in~\cite[p. 145]{LCG} and was proved recently by the authors in~\cite{DT2}.
One of the reviewers of the work in~\cite{DT2} gave us a suggestion concerning future work, namely a characterization of $r$-regular sets
in terms of the Lipschitz regularity of normal vector fields to the boundary.
In this paper we provide this characterization.

\section{Preliminary Definitions}
\label{sec:1}
In this section we summarize the basic smoothness concepts used in the proofs.
For more details, we refer the reader to~\cite{MS}.
Let $\mathscr{L}(\R^n,\R^m)$ denote the vector space of linear maps
$L:\R^n\to\R^m$. Given  some open set $U\subseteq\R^ n$ in the Euclidean space $\R^ n$,
a map $f:U\to\R^m$ is said to be of {\em class} $C^1$ if and only if there is
a continuous function $Df:U\to \mathcal{L}(\R^n,\R^m)$ such that for $x\in U$, as $v\to 0$ in $\R^ n$, one has
$f(x+v)=f(x)+Df_x(v) +o(\nrm{v})$, i.e., given $x\in U$ and  $\varepsilon>0$ there is $\delta>0$ such that
for every $v\in \R^n$ with $\nrm{v}\leq\delta$ and $x+v\in U$ one has
$\nrm{f(x+v)-f(x)-Df_x(v)}\leq \varepsilon\,\nrm{v}$.
For each $x\in U$, the linear map $Df_x$ is unique and called the {\em differential} of $f$ at point $x$.
The map $Df:U\to \mathcal{L}(\R^n,\R^m)$ is called the {\em total derivative} of $f$.
A map $f: U \to \R^m$ is said to be of {\em class} $C^{1+{\rm Lip}}$ if and only if $f$ is of class $C^1$
and $Df:U\to \mathcal{L}(\R^n,\R^m)$ is Lipschitz. We shall denote by ${\rm Lip}(f)$ the Lipschitz constant of a 
Lipschitz map $f$.

Radamacher's theorem states that a Lipschtiz map is differentiable almost everywhere (see~\cite[Section 3.1.6.]{F}).

\begin{teor}(Rademacher's theorem)
\label{Rademacher} If $f: \R^n \to \R^m$ is Lipschitz, then $f$ is differentiable at almost all points of $\R^n$.
\end{teor}

This work deals essentially with class $C^{1+{\rm Lip}}$ differentiability,
but since we mention $C^2$ smoothness we also provide the definition.
A map $f:U\to\R^m$ is said to be of {\em class} $C^2$ if and only if $f$ is of class $C^ 1$ and
$Df:U\to \mathcal{L}(\R^n,\R^m)$ is also of class $C^ 1$.
The total derivative $D(Df)$ is a function with values in the vector
space $\mathcal{L}(\R^n,\mathcal{L}(\R^n,\R^m))$, which can be identified with
the space of symmetric bilinear forms $B:\R^n\times\R^n\to\R^m$.
Hence, $D(Df)_x$ identifies with a symmetric bilinear form $D^ 2 f_x:\R^n\times\R^n\to\R^m$,
referred as the second order differential of $f$ at point $x$.
It follows from the definition of a class $C^2$ map that for any $x\in U$,
$f(x+v)=f(x)+Df_x(v) + \frac{1}{2}\, D^ 2 f_x(v,v) + o(\nrm{v}^ 2)$,
as $v\to 0$ in $\R^ n$. Equivalently,
a function is of class $C^2$ if and only if it has continuous partial
derivatives up to the second order.

Given open sets $U,V\subseteq \R^n$, a map $h:U\to V$ is called
a {\em diffeomorphism} of class $C^k$ if and only if $h:U\to V$
is bijective, while the maps $h:U\to V$ and $h^{-1}:V\to U$ are of class $C^k$.

Take $0\leq d \leq n$.
A subset $M\subseteq \R^n$ is called a {\em manifold} of class $C^ k$ and dimension $d$ if and only if
for every $x\in M$ there are open sets $U,V\in\R^n$, with $x\in U$ and $0\in V$,
and  a diffeomorphism of class $C^k$ $h:U\to V$ such that
$h(M\cap U)=(\R^d\times\{0\})\cap V$. This means that $M$ is locally equivalent
`up to a diffeomorphism' to the $d$-dimensional subspace $\R^d\times\{0\}$ of $\R^ n$.
The number $n-d$ is called the {\em codimension} of $M$ in $\R^n$
A practical way of proving that a subset $M\subseteq \R^n$ is a manifold is the following
{\em pre-image theorem}.

\begin{teor} Given an open set $V\subseteq \R^n$ and a map $f:V\to\R^{m}$ of class $C^k$,
with $0\leq m\leq n$, assume that $Df_x:\R^n\to\R^m$ is surjective, for every $x\in V$
such that $f(x)=0$. Then $f^{-1}(0)=\{\, x\in V\,:\, f(x)=0\,\}$ is a manifold of class
$C^k$ and dimension $d=n-m$.
\end{teor}

When $m=1$ the differential of $f$ at $x$ can be written as
$Df_x(v)=\langle \nabla f(x), v\rangle$, where $\nabla f(x)$ denotes the gradient vector
$\nabla f(x)=\left(\frac{\partial f}{\partial x_1}(x),\ldots, \frac{\partial f}{\partial x_n}(x)\right)$.
In this case, the pre-image theorem specializes as follows.

\begin{coro} \label{pre-image}
Given a map $f:V\to\R$ of class $C^k$ on an open set $V$,
define $M=f^{-1}(0)$. If $\nabla f(x)\neq 0$  for every $x\in M$,
then $M$ is a codimension one manifold of class $C^k$.
\end{coro}

This result can be easily generalized to $C^{1+{\rm Lip}}$ manifolds.


\begin{prop}\label{pre-image-C1lip}
Given a map $f:V\to\R$ of class $C^{1+{\rm Lip}}$ on an open set $V$,
define $M=f^{-1}(0)$. If $\nabla f(x)\neq 0$  for every $x\in M$,
then $M$ is a codimension one manifold of class $C^{1+{\rm Lip}}$.
In particular it admits
an atlas consisting of charts  (local diffeomorphisms) of class
$C^{1+{\rm Lip}}$.
\end{prop}

\begin{proof}
Given a point $p\in M=f^{-1}(0)$, consider a surjective linear map
$P:\R^n\to\R^{n-1}$ such that ${\rm Ker}(P)^\perp={\rm Ker}(D f_p)$.
Defining $\Phi: V\to\R^n$, $\Phi(x)=(f(x),P(x))$, $D\Phi_p:\R^n\to\R^n$ is an isomorphism.
Thus $\Phi$ is a mapping of class $C^{1+{\rm Lip}}$ and a 
 $C^1$ diffeomorphsim on some small neighbourhood $U$ of $p$.
 The inverse map $\Phi^{-1}:\Phi(U)\to U$ is of class $C^1$
 with derivative
 $D \Phi^{-1}_x = \left( D \Phi_x \right)^{-1}$, and since the matrix inversion is a 
 Lipschitz mapping,  $\Phi^{-1}$ is of class $C^{1+{\rm Lip}}$. 
 Denoting by $\pi:\R\times\R^{n-1}\to\R^{n-1}$ the canonical projection,
 $\pi(t,x)=x$, the mapping $\phi=\pi\circ\Phi\vert_{M\cap U}:M\cap U\to \R^{n-1}$
 is a local chart of class $C^{1+{\rm Lip}}$ for $M$.
\end{proof}

\section{Main Statement}
\label{sec:2}

The class of $r$-regular sets was independently introduced in~\cite{P} and~\cite{Se}.
This class is also referred in \cite{A,BB,KS,LCG,MKS,S1,S,SK,SLS,ST}. Although the details of the definitions in these papers are different, the described class is essentially the same and can be defined as follows.
Denote by $B(x,r)$ the Euclidean open ball with center $x\in\R^n$ and radius $r$.
Fix a positive number $r$ and define $\Uscr_r$
as the set of all connected unions of balls $B(x,r')$ with radius $r'\geq r$. Note that, as any ball $B(x,r')$ with
radius $r'\geq r$ is itself a union of balls of radius $r$, any set in $\Uscr_r$
is a union of balls of radius $r$.

\begin{defi}
An open set $U \subseteq \R^n$ is said to be  $r$-regular if and only if $U \in \Uscr_r$ and $\overline{U}^c \in \Uscr_r$.
\end{defi}

\bigskip

Let $U \subseteq \R^n$ be an 
open set.

\begin{defi}
A normal vector field along $\partial U$ is any vector
field $\eta:\partial U\to\R^n$ such that for each $x\in\partial U$,
$$ \langle \eta(x),y-x\rangle =o(\nrm{x-y}) \; \text{ as } \; y\to x \; \text{ in } \; \partial U\;.$$
\end{defi}

\bigskip

If $\eta$ is a normal vector field, on a compact neighbourhood of any point $x\in\partial U$  there is a monotonic continuous function $\rho:\R^+_0\to\R^+_0$,
with $\rho(0)=0$, such that
$$ \langle \eta(x),y-x\rangle \leq \rho(\nrm{x-y})\,\nrm{x-y}, \; \forall\; x,y\in\partial U\;.$$

\bigskip

The {\em intrinsic metric} on $\partial U$ (see e.g.~\cite{BBI}), denoted by $d_{\partial U}$, is defined as follows:
given  $x,y \in \partial U$ and  $\epsilon>0$, let
$$d_\epsilon(x,y)=\inf_{\begin{array}{c}
                 x_0=x,\, x_n=y,\,\\
                x_i\in\partial U, \, \nrm{ x_{i+1}- x_i} < \epsilon
                  \end{array}} \sum_{i=0}^{n-1}
                  \nrm{ x_{i+1}- x_i}\; . $$
Now define
$$d_{\partial U}(x,y)=\sup_{\epsilon>0} d_{\epsilon}(x,y).$$

\bigskip

The aim for the rest of this paper is the proof of the following characterization of $r$-regularity:

\begin{teor}\label{main} Let $U\subseteq\R^n$ be an 
open set.
Then $U$ is $r$-regular\, if and only if 
\begin{enumerate}
\item[$(1)$]\, there is $\eta:\partial U\to\R^n$
such that
\begin{enumerate}
\item[$\mbox{\rm (i)}$]\, $\eta$ is a normal vector field  along $\partial U$,
\item[$\mbox{\rm (ii)}$]\, $\nrm{\eta(x)}=r$, for every $x\in\partial U$,
\item[$\mbox{\rm (iii)}$]\, ${\rm Lip}(\eta)\leq 1$,
\end{enumerate}
\item[$(2)$]\, $d_{\partial U}(x,y) \geq \frac{\pi}{2}\, \nrm{x-y}$ implies $\nrm{x-y} \geq 2\,r$, for all $x,y \in \partial U$.
\end{enumerate}
\end{teor}

\section{Local Characterization of Regularity}
\label{sec:3}
In this section we
prove that an $r$-regular set
admits a Lipschitz normal vector field along its boundary (see Proposition~\ref{localmain}) .
We also prove that the existence of a Lipschitz normal vector field along $\partial U$ ensures that $U$ is locally $r$-regular (see Proposition~\ref{localnormal}) .

\bigskip

\begin{prop}\label{localmain} Let $U\subseteq\R^n$ be an $r$-regular  set.
Then  there is $\eta:\partial U\to\R^n$
such that
\begin{enumerate}
\item[$\mbox{\rm (i)}$]\, $\eta$ is a normal vector field  along $\partial U$,
\item[$\mbox{\rm (ii)}$]\, $\nrm{\eta(x)}=r$, for every $x\in\partial U$,
\item[$\mbox{\rm (iii)}$]\, ${\rm Lip}(\eta)\leq 1$.
\end{enumerate}
\end{prop}

\begin{proof}
Since $U$ is $r$-regular, items $\mbox{\rm (i)}$, $\mbox{\rm (ii)}$ and $\mbox{\rm (iii)}$ follow from the previous paper~\cite{DT2}: items $\mbox{\rm (i)}$ and $\mbox{\rm (ii)}$
follow from \cite[Proposition 5]{DT2} and item $\mbox{\rm (iii)}$ follows from \cite[Lemma 3(1)]{DT2} with $v=\eta(x)$, $w=\eta(y)$ and $u=y-x$.
\end{proof}

\begin{defi}
An open set $U \subseteq \R^n$ is said to be locally $r$-regular if and only if
for every $x \in \partial U$  there are two balls $B(x_1,r)$ and $B(x_2,r)$
tangent at $x$ such that $B(x_1,r)\cap B(x,\varepsilon)\subseteq U$ and
$B(x_2,r)\cap B(x,\varepsilon)\cap \overline{U}=\emptyset$, for some $\varepsilon>0$.
\end{defi}

\bigskip

\begin{prop}\label{localnormal} Let $U\subseteq \R^n$ be an  open set and let  $\eta:\partial U\to\R^n$
be such that
\begin{enumerate}
\item[$\mbox{\rm (i)}$]\, $\eta$ is a normal vector field  along $\partial U$,
\item[$\mbox{\rm (ii)}$]\, $\nrm{\eta(x)}=r$, for every $x\in\partial U$,
\item[$\mbox{\rm (iii)}$]\, ${\rm Lip}(\eta)\leq 1$.
\end{enumerate}
Then $U$ is locally $r$-regular.
\end{prop}

The rest of this section is dedicated to the proof of Proposition~\ref{localnormal}.
We assume that $\eta:\partial U\to\R^n$ is a  normal vector field  along $\partial U$
such that $\nrm{\eta(x)}=r$, for every $x\in\partial U$,
and  ${\rm Lip}(\eta)\leq 1$.
First we shall prove some auxiliary lemmas.

\bigskip




\begin{lema}\label{quadrado} Given $x\in\partial U$ we have that
$$\langle \eta(x), \eta(y)-\eta(x)\rangle = o(\nrm{x-y}) \; \text{ as } \; y\to x \; \text{ in } \; \partial U\;.$$
Moreover, given $t\in (-1/2,1/2)$, we have that
 $$\nrm{x+t\,\eta(x)-(y+t\,\eta(y))}\geq \sqrt{1-2|t|} \,\nrm{x-y},$$
 for any $x,y \in \partial U$.
\end{lema}

\begin{proof}
Since
\begin{align*}
&\nrm{\eta(x) -\eta(y)}^2  =  \langle \eta(x)-\eta(y), \eta(x)-\eta(y) \rangle\\
& \hspace{0.4cm} = 2\,\langle \eta(x), \eta(x)-\eta(y) \rangle - \langle \eta(x)+\eta(y), \eta(x)-\eta(y) \rangle\\
& \hspace{0.4cm}  = 2\,\langle \eta(x), \eta(x)-\eta(y) \rangle - \underbrace{(\nrm{\eta(x)}^2-\nrm{\eta(y)}^2 )}_{=0}
\end{align*}
we get that
$$ \langle \eta(x), \eta(x)-\eta(y) \rangle =\frac{1}{2}\, \nrm{\eta(x)-\eta(y)}^2 \leq \frac{1}{2}\, \nrm{x-y}^2,$$
where the last inequality follows because ${\rm Lip}(\eta)\leq 1$. Therefore,
$$\langle \eta(x), \eta(y)-\eta(x)\rangle = o(\nrm{x-y}) \; (y\to x).$$

Finally, given $t\in (-1/2,1/2)$, we have that
\begin{align*}
&\nrm{x+t\,\eta(x)-y-t\,\eta(y)}^ 2  = \\
& \hspace{0.4cm} = \nrm{t\,\eta(x)-t\,\eta(y)}^2 +\nrm{x-y}^ 2 \\
&  \hspace{0.8cm} + 2\,\langle x-y,t\,\eta(x)-t\,\eta(y) \rangle \\
& \hspace{0.4cm} \geq  \nrm{x-y}^2 - 2 |t|\, \nrm{x-y}\nrm{\eta(x)-\eta(y)} \\
& \hspace{0.4cm} \geq  \nrm{x-y}^2 - 2 |t|\, \nrm{x-y}^2 \\
& \hspace{0.4cm} =  \nrm{x-y}^2 \, \left(1 - 2\, |t| \right),
\end{align*}
which implies that
\begin{align*}
\nrm{x+t\,\eta(x)-y-t\,\eta(y)} \geq \sqrt{1 - 2\, |t|} \, \nrm{x-y},
\end{align*}
 and completes the proof.
\end{proof}

\bigskip

It follows that on a compact neighbourhood of every point $x\in\partial U$  
there is a monotonic continuous function $\tilde{\rho}:\R^+_0\to\R^+_0$,
with $\rho(0)=0$, such that
$$ \langle \eta(x),\eta(y)-\eta(x)\rangle \leq \tilde{\rho}(\nrm{x-y})\,\nrm{x-y}, \; \forall\; x,y\in\partial U\;.$$

\bigskip

For each $\delta>0$ we define the
 {\em $\delta$-tubular neighbourhood} of $\partial U$,
$N_\delta = \{\, x\in\R^n\,:\, d(x,\partial U)<\delta\,\}$.

\bigskip

\begin{figure}[h]
\begin{center}
\includegraphics[width=0.5\textwidth]{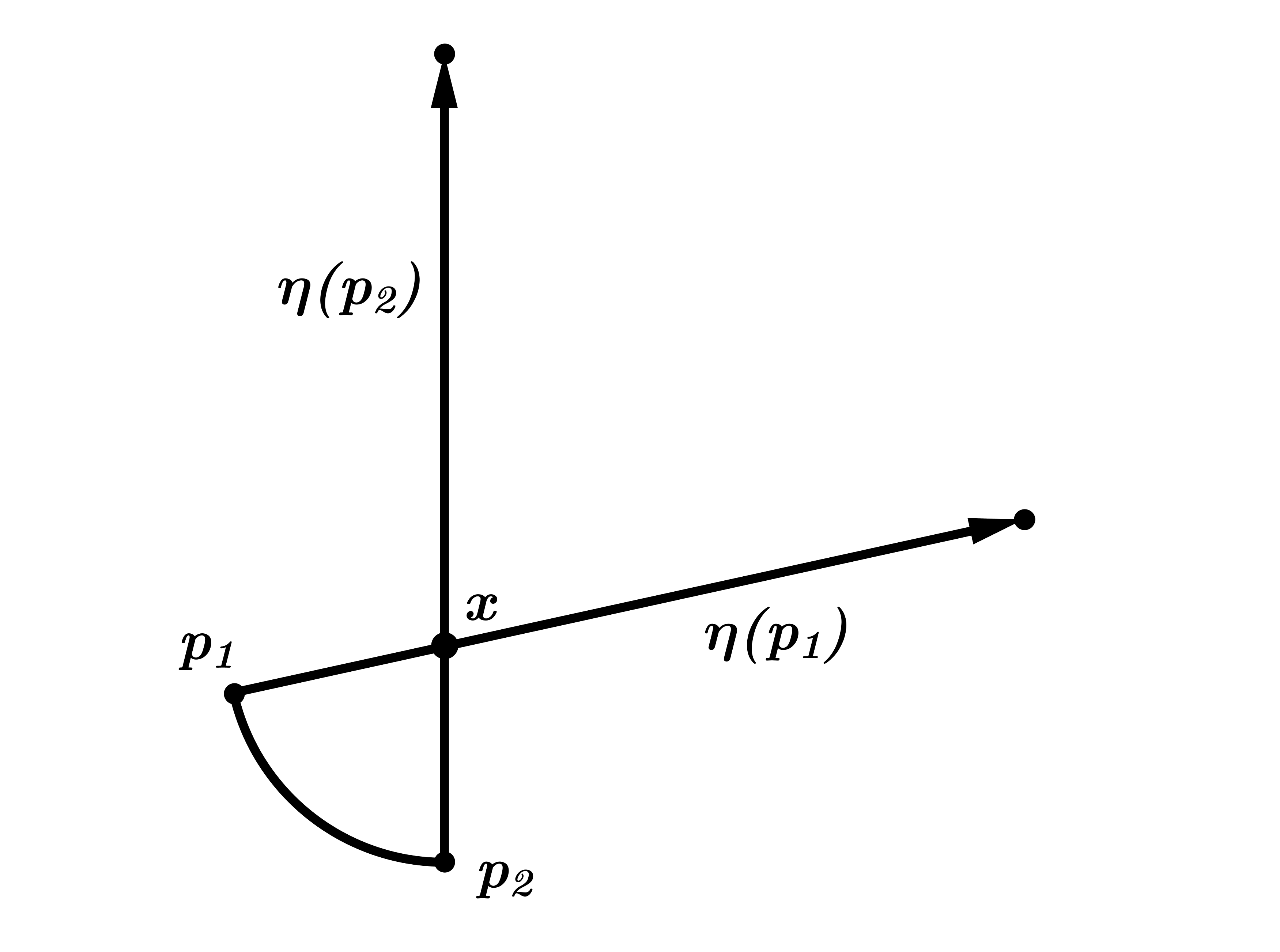}
\end{center}
\caption{Uniqueness of the projection}
\label{fig:1}
\end{figure}

\begin{lema}\label{projection}
Given $x\in N_{r/2}(\partial U)$ there is a unique point $p \in\partial U$
such that $\nrm{x-p}=d(x,\partial U)$.
\end{lema}

\begin{proof}
Assume there are two points $p_1\neq p_2$ in $\partial U$
such that $\nrm{x-p_1}=\nrm{x-p_2}=d(x,\partial U)$, where $x\in N_{r/2}(\partial U)$.
Then since $\nrm{x-p_1}=\nrm{x-p_2}<r/2$, we have (see Fig.~\ref{fig:1}),
\begin{align*}
\nrm{\eta(p_1)-\eta(p_2)} &= \frac{r}{d(x,\partial U)} \nrm{(x-p_1)-(x-p_2)} \\
& =  \frac{r}{d(x,\partial U)} \nrm{p_1-p_2} \\
& \geq \frac{r}{r/2} \nrm{p_1-p_2} = 2\,\nrm{p_1-p_2}\, ,
\end{align*}
which contradicts
${\rm Lip}(\eta)\leq 1$, unless $p_1=p_2$.
\end{proof}

\bigskip

By the previous lemma we can define a projection $\pi:N_{r/2}(\partial U)\to \partial U$
such that $\nrm{x-\pi(x)}=d(x,\partial U)$, for every $x\in N_{r/2}(\partial U)$.

\bigskip

\begin{lema}\label{lipschitz}
Given any \, $0 < s  <  \frac{1}{2}$, \, we have that the mapping $\pi: N_{sr}(\partial U) \to \partial U$ is a Lipschitz projection with
${\rm Lip}(\pi)\leq \frac{1}{\sqrt{1-2 s}}$,
and for every $x\in\partial U$,
$$\pi^ {-1}(x)=\{\, x+t\,\eta(x)\,:\, t\in (-s,s)\,\}\;.$$
\end{lema}

\begin{proof}
The second part is clear, hence we only need to prove that $\pi$ is Lipschitz.
Given $x, y\in N_{sr}(\partial U)$, we have $x=\pi(x)+t_1\,\eta(\pi(x))$
and $y= \pi(y)+t_2\,\eta(\pi(y))$, with $t_1,t_2\in (-s,s)$.
Let $x'$, $y'$ be such that
$x'= \pi(x)+t\,\eta(\pi(x))$ and $y'= \pi(y)+t\,\eta(\pi(y))$, for some $t\,\in (-s,s)$,
and such that $\nrm{x'-y'} \leq \nrm{x-y}$.
Clearly, $\pi(x)=\pi(x')$ and $\pi(y)=\pi(y')$.
Using the almost orthogonality relation of Lemma 1, we obtain that
\begin{align*}
&\nrm{\pi(x)-\pi(y)} = \nrm{\pi(x')-\pi(y')} \\
& \hspace{0.4cm} \leq \frac{1}{\sqrt{1-2|t|}}  \nrm{(\pi(x)+t\,\eta(\pi(x)))-(\pi(y)+t\,\eta(\pi(y))}  \\
& \hspace{0.4cm} =  \frac{1}{\sqrt{1-2|t|}} \nrm{x'-y'}  \leq \frac{1}{\sqrt{1-2|t|}} \nrm{x-y}\\
& \hspace{0.4cm} \leq \frac{1}{\sqrt{1-2 s}} \nrm{x-y}.
\end{align*}
\end{proof}

\bigskip

Define now the function $f:N_{r/2}(\partial U)\to \R$,
$$ f(x)=\langle x-\pi(x),\eta(\pi(x)) \rangle \;.$$

\bigskip

\begin{prop}\label{fC1} The function $f$ is of class $C^1$ with differential given by
$$Df_x(v)= \langle v,\eta(\pi(x)) \rangle .$$
\end{prop}

\begin{proof} Next argument is a straightforward adaptation of the proof of Proposition 8 in~\cite{DT2},
which we include here for the sake of completeness.
We must show that  one has  for every  $x \in N_{r/2}(\partial U)$  
$$f(y)-f(x)-\langle y-x,\eta(\pi(x))\rangle = o(\nrm{x-y})\; (y\to x)$$
i.e., given $x \in N_{r/2}(\partial U)$ and $\epsilon>0$ there is $\delta >0$ such that for every $y \in N_{r/2}(\partial U)$ with
$\nrm{x-y} \leq \delta$ one has
$$\Mod{f(y)-f(x)-\langle y-x,\eta(\pi(x))\rangle} \leq \epsilon\, \nrm{x-y}\;.$$
To simplify the notation we shall omit some parentheses, writing for instance  $\pi x$ instead of $\pi(x)$.

\begin{align*}
\begin{split}
&    \Mod{f(y)-f(x)-\langle y-x,\eta(\pi x)\rangle}  = \\
& =  \Mod{\langle y-\pi y,\eta(\pi y)\rangle - \langle x-\pi x,\eta(\pi x)\rangle -\langle y-x,\eta(\pi x)\rangle} \\
& =  \Mod{\langle y-\pi y,\eta(\pi y)\rangle - \langle y-\pi x,\eta(\pi x)\rangle} \\
& =  \Mod{\langle y-\pi y,\eta(\pi y)\rangle - \langle y-\pi x,\eta(\pi y)\rangle} \\
&         \hspace{0.4cm}  + \Mod{\langle y-\pi x,\eta(\pi y)\rangle - \langle y-\pi x,\eta(\pi x)\rangle}\\
& = \Mod{\langle \pi x-\pi y,\eta(\pi y)\rangle}  + \Mod{\langle y- \pi x,\eta(\pi y)-\eta(\pi x)\rangle}  \\
& = \Mod{\langle \pi x-\pi y,\eta(\pi y)\rangle}  + \Mod{\langle y- x,\eta(\pi y)-\eta(\pi x)\rangle}  \\
&        \hspace{0.4cm} + \Mod{\langle x-\pi x,\eta(\pi y)-\eta(\pi x)\rangle}  \\
& =  \Mod{\langle \pi x-\pi y,\eta(\pi y)\rangle}  + \Mod{\langle y- x,\eta(\pi y)-\eta(\pi x)\rangle}  \\
&         \hspace{0.4cm} + \frac{\nrm{x-\pi x}}{r}\, \Mod{\langle \eta(\pi x) ,\eta(\pi y)-\eta(\pi x)\rangle}   \\
& \leq \rho(\nrm{\pi x - \pi y})\nrm{\pi x - \pi y} + \nrm{x-y}\,\nrm{\eta(\pi x)-\eta(\pi y)} \\
&  \hspace{0.4cm} + \tilde{\rho}(\nrm{\pi x - \pi y})\nrm{\pi x - \pi y} \\
& \leq C \rho(\nrm{\pi x - \pi y})\nrm{x - y} + C \nrm{x-y}^2 + C \tilde{\rho}(\nrm{ x - y})\\
& \leq C\,\nrm{x-y}\,(\rho(\nrm{\pi x - \pi y}) + \nrm{x-y} + \tilde{\rho}(\nrm{\pi x - \pi y}))  \\
& \leq C \, \nrm{x-y}\, \hat{\rho}(\nrm{x-y})	\;,\\
\end{split}
\end{align*}
where $\hat{\rho}(t)=t+\rho(t) + \tilde{\rho}(t)$.
We observe that in the penultimate step we have used the fact that both $\eta$ and $\pi$ are Lipschitz on $N_{sr}$, provided that
$s \in (0,1/2)$, and the constant $C=C_s$ is given explicitly by $C_s=\frac{1}{\sqrt{1-2 s}}$.
\end{proof}

\bigskip

\begin{rmk}
In Proposition~\ref{fC1}, function $f$ is indeed of class $C^{1+{\rm Lip}}$.
Just note that, since  $\eta$ and $\pi$ are Lipschitz, the differential $Df$ is also Lipschitz.
\end{rmk}

\bigskip

We are now ready to prove Proposition~\ref{localnormal}, i.e., to show that 
the level set $\partial U=f^{-1}(0)$ is locally bounded between two spheres of radius $r$ tangent at $x$, for any $x \in \partial U$.

\begin{proof}[of Proposition~\ref{localnormal}]
 By the mean value theorem, for any $x, x+v \in  N_{r/2}(\partial U)$ we have that
$$f(x+v)- f(x)-Df_x(v)  = \int_0^1 Df_{x+t\,v}(v)\,dt - Df_x(v).$$
But since, for any $x,y \in N_{r/2}(\partial U)$, we have that
\begin{align*}
\begin{split}
\Mod{Df_x(v)-Df_y(v)}  & = \Mod{\langle v,\eta(\pi x)\rangle - \langle v,\eta(\pi y)\rangle}\\
& =  \Mod{\langle v,\eta(\pi x) - \eta(\pi y)\rangle} \\
& \leq \nrm{v}\nrm{\pi x - \pi y} \\
& \leq  {\rm Lip}(\pi) \nrm{v} \nrm{x-y},\\
\end{split}
\end{align*}
it follows that
\begin{align*}
\begin{split}
& \Mod{f(x+v)- f(x)-Df_x(v)} = \\
& \hspace{0.4cm} = \Mod{\int_0^1 (Df_{x+t\,v}(v) - Df_x(v)) \, dt}\\
& \hspace{0.4cm} \leq \int_0^1 \Mod{Df_{x+t\,v}(v) - Df_x(v)} \, dt \\
& \hspace{0.4cm} \leq \int_0^1 {\rm Lip}(\pi) \nrm{v} \nrm{x+t\,v-x}\, dt \\
& \hspace{0.4cm} =  {\rm Lip}(\pi) \nrm{v}^2 \int_0^1 t\,dt\\
& \hspace{0.4cm} =  \frac{1}{2}\,{\rm Lip}(\pi) \nrm{v}^2.\\
\end{split}
\end{align*}
Therefore, given $x \in \partial U$ and $x+v \in N_{sr}$, for some $s \in (0,1/2)$, we get that
\begin{align*}
\begin{split}
& \Mod{f(x+v)-f(x)-Df_x(v)} = \\
& \hspace{0.4cm} = \Mod{f(x+v)-Df_x(v)} \leq \frac{C_s}{2}\,\nrm{v}^2,\\
\end{split}
\end{align*}
or, equivalently,
\begin{equation}\label{balls}
\Mod{f(x+v)-\langle v, \eta(x) \rangle} \leq  \frac{C_s}{2}\,\nrm{v}^2,
\end{equation}
where $C_s=\frac{1}{\sqrt{1-2 s}}$.

Now define the set $A(x)$ of vectors $v\in\R^n$ such that 
$$ \nrm{(x+v)-\left(x \pm \frac{\eta(x)}{C_s}\right)}^2 < \left(\frac{r}{C_s}\right)^2 \;,$$
for one of the signs $+$ or $-$.
Note that $x+v$ belongs to one of the two balls 
$B\left(x \pm \frac{\eta(x)}{C_s}, \frac{r}{C_s}\right)$, which are tangent at $x$,\,
if and only if \, $v\in A(x)$.
We claim that $v\in A(x)$ implies that
$$\left|\langle v, \eta(x) \rangle \right| > \frac{C_s}{2}\,\nrm{v}^2.$$
In fact, first note that
\begin{align*}
\begin{split}
& \nrm{(x+v)-\left(x \pm \frac{\eta(x)}{C_s}\right)}^2  \geq \\
& \hspace{1.0cm} \nrm{v}^2 + \left(\frac{\nrm{\eta(x)}}{C_s}\right)^2 - \frac{2}{C_s} \left|\langle v, \eta(x) \rangle \right|. \\
\end{split}
\end{align*}
Consequently, if $v \in A(x)$, then the following inequality holds
$$  \left(\frac{r}{C_s}\right)^2 >   \nrm{v}^2 + \left(\frac{\nrm{\eta(x)}}{C_s}\right)^2 - \frac{2}{C_s} \left|\langle v, \eta(x) \rangle \right| $$
which is equivalent to
$$ \displaystyle{\left|\langle v, \eta(x) \rangle \right|  >  \frac{C_s}{2}\,\nrm{v}^2}.$$
Therefore, if $v \in A(x)$, then inequality~(\ref{balls}) implies that
$$\displaystyle{\left|f(x+v) \right| \geq \left| \langle v, \eta(x) \rangle \right|- \frac{C_s}{2} \nrm{v}^2 >0}.$$
Hence,
$$\left\{ x+v:\, v \in A(x)\right\} \cap f^{-1}(0) = \emptyset$$
and thus, we may conclude that locally the level set $\partial U=f^{-1}(0)$ is bounded between two  balls of radius $r/{C_s}$
tangent at $x$.
Notice that the constant $C_s$ gets close to $1$ as $s \rightarrow 0$.
Therefore, decreasing the neighbourhood $N_{sr}$ of $\partial U$, we can take the radius of the balls arbitrarily close to $r$.
\end{proof}

\section{An Example}
\label{sec:4}
Let $U \subseteq \R^n$ be an open set.
In this section we show that the existence of a normal vector field of constant norm $r$ and Lipschitz constant $1$
along $\partial U$ (given by items $(1)$ $\mbox{\rm (i)}$, $\mbox{\rm (ii)}$ and  $\mbox{\rm (iii)}$ of Theorem~\ref{main}) is not enough to guarantee that $U$ is (globally) $r$-regular.
Indeed, consider a body $U$, as sketched in Fig.~\ref{fig:2}, but in a large scale so that the curvature is small, but the connection between the balloons is still
very thin. Then the Lipschitz assumption on the ``normal vector field" is satisfied, but $U$ is not $r$-regular because of the narrow connection.

\begin{figure}[h]
\begin{center}
\includegraphics[width=0.5\textwidth]{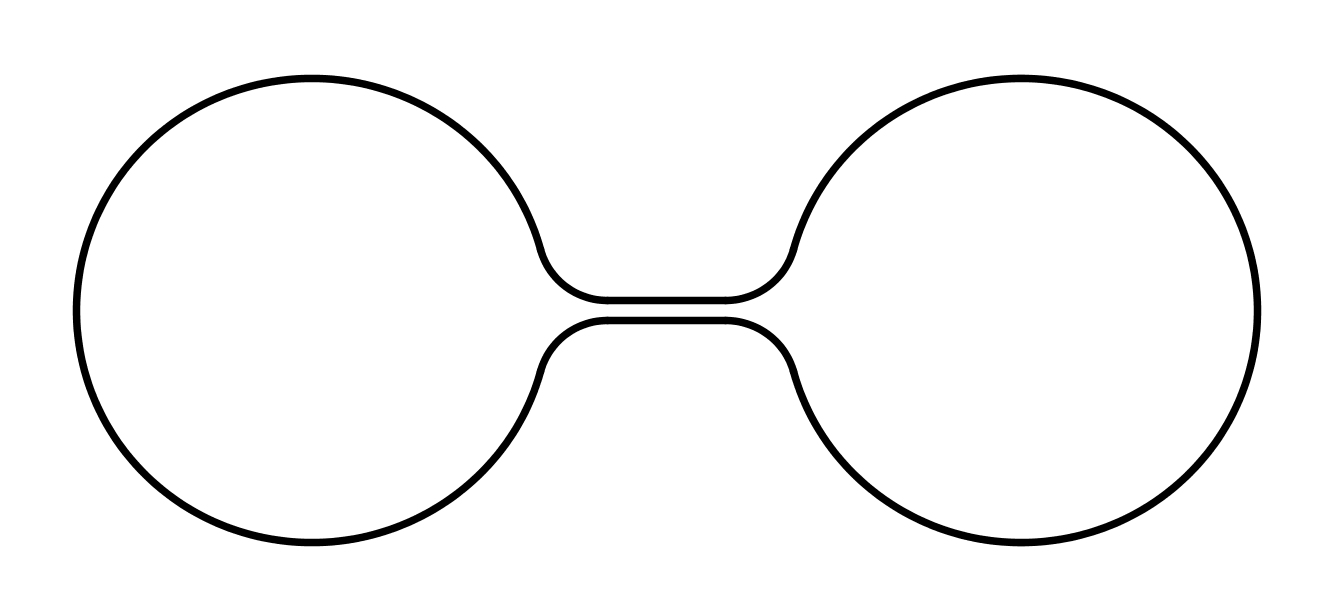}
\end{center}
\caption{A non-regular set that satisfies the Lipschitz assumption on the normal vector field}
\label{fig:2}
\end{figure}

Therefore, we need an extra assumption to ensure that $U$ is (globally) $r$-regular.
This extra assumption (given by item $(2)$ of Theorem~\ref{main}) is that for every $x,y \in \partial U$,
$d_{\partial U}(x,y)\geq \frac{\pi}{2}\, d(x,y)$ implies $d(x,y)\geq 2\,r$. The bound $\frac{\pi}{2}$ corresponds to the maximum ratio
between these two distances when $U$ is an Euclidean ball.

\section{$C^{1+{\rm Lip}}$-geodesics}
\label{sec:5}
In the proof of the ``if" part of Theorem~\ref{main} we shall use some tools and results from Riemannian Geometry.
The natural scope of this theory is that of smooth manifolds of at least class $C^2$. The results we need can easily be generalized to  $C^{1+{\rm Lip}}$ manifolds, but since we are not aware of any suitable reference, for the sake of completeness,  we include this section  where we prove the needed generalizations.

\bigskip

From now on $M\subseteq \R^n$ will denote a class $C^{1+{\rm Lip}}$ connected manifold of dimension $n-1$,
defined as a regular level set of some function of class $C^{1+{\rm Lip}}$. By Proposition~\ref{pre-image-C1lip},
given $p,q\in M$ there are  $C^{1+{\rm Lip}}$ curves
$\gamma:[a,b]\to M$ such that $\gamma(a)=p$ and $\gamma(b)=q$.
Hence the  {\em loop space} $\Omega=\Omega_{p,q}(a,b)$ of all
such curves is non empty. This loop space is a subset of
the normed vector space
$C^{1+{\rm Lip}}([a,b],\R^n)$ endowed with the norm
$$\nrm{\gamma}_{C^1}= \max_{t\in[a,b]} \nrm{\gamma(t)}  + \max_{t\in[a,b]} \nrm{\gamma'(t)} \;.$$
We topologize $\Omega$ as a subspace of this normed space.
The {\em energy} and {\em length} of a curve
$\gamma\in\Omega$ are respectively defined to be
$$ E(\gamma)=\int_a^ b \nrm{\gamma'(t)}^ 2\, dt \quad \text{ and }
\quad L(\gamma)=\int_a^ b \nrm{\gamma'(t)}\, dt\;.
$$

Fix a  curve $\gamma_0\in\Omega$.
We call {\em variation of $\gamma_0$} to any continuous function
$h:(-\delta,\delta)\times [a,b]\to M$ such that
\begin{enumerate}
\item[(1)] $h(s,t)$ is Lipschitz on $(-\delta,\delta)\times [a,b]$,
\item[(2)] $h(0,t)=\gamma_0(t)$ for all $t\in [a,b]$,
\item[(3)] $t\mapsto h(s,t)$ is of class $C^{1+{\rm Lip}}$
for every $s\in (-\delta,\delta)$,
\item[(4)] the function $\zeta(t):=\frac{dh}{ds}(0,t)$ is
Lipschitz in $[a,b]$.
\end{enumerate}
We shall also say that $h$ is a {\em variation of $\gamma_0$
along  $\zeta$}.
The variation $h$ is said to be {\em proper} if furthermore
\begin{enumerate}
\item[(5)] $h(s,a)=p$  and $h(s,b)=q$  for all $s\in (-\delta,\delta)$.
\end{enumerate}
The variation $h$ determines a curve $\overline{h}:(-\delta,\delta)\to \Omega$ such that $\overline{h}(0)=\gamma_0$.
The  vector field $\zeta$ along $\gamma_0$  is the formal derivative $\overline{h}'(0)=\zeta$ of this curve.
It is a section of the  vector bundle $\gamma_0^\ast TM$ over $[a,b]$, which has fiber $T_{\gamma_0(t)} M$ at 
the base point $t\in [a,b]$. We shall also refer to $\zeta$ as a {\em tangent vector field}, to emphasize 
that it is tangent to the hypersurface $M$.
If $h$ is proper variation then the vector field $\zeta$ is also proper in the
sense that $\zeta(a)=\zeta(b)=0$.

\bigskip

The existence of variations of a curve $\gamma_0\in\Omega$
is usually proven through the {\em exponential map} of the
Riemannian manifold $M$. We avoid these technicalities working in a local chart.
Given  $\varepsilon>0$, two vector fields $\zeta_1, \zeta_2:[a,b]\to\R^n$ are said to be
{\em  $\varepsilon$-close }\, when $\nrm{\zeta_1(t)-\zeta_2(t)}\leq \varepsilon$ for all $t\in [a,b]$.

\begin{prop}\label{var:exist}
Given  $\varepsilon>0$, a curve $\gamma_0\in\Omega$ contained in the domain of a single chart of $M$,  
  and any continuous proper vector field $\zeta:[a,b]\to\R^n$ tangent to $M$  
along $\gamma_0$, there is at least one proper
  variation of $\gamma_0$ along a  Lipschitz vector field
  $\varepsilon$-close to $\zeta$.
\end{prop}

\begin{proof}
Fix a curve $\gamma_0 \in \Omega$ for which there is a local chart  $\phi:U\subseteq M\to \R^{n-1}$ 
of class $C^{1+{\rm Lip}}$ such that $\gamma_0([a,b])\subset U$.
Define
$$\tilde{\zeta}(t):=D \phi_{\gamma_0(t)} \zeta(t),$$
and
$$\tilde{\gamma}_0(t) := \phi(\gamma_0(t)).$$
Let $\tilde{\zeta}^{*}$ be a proper $C^\infty$-vector field $\varepsilon$-close to $\tilde{\zeta}$.
We define the function $\tilde{h}: (-\delta,\delta)\times [a,b]\to \R^{n-1}$ by
$$\tilde{h}(s,t):=\tilde{\gamma}_0(t) + s \,\tilde{\zeta}^{*}(t).$$
Finally, let
$$
h(s,t):=\phi^{-1} \tilde{h} (s,t) = \phi^{-1} (\tilde{\gamma}_0(t) + s \, \tilde{\zeta}^{*} (t)).$$
We have that
\begin{enumerate}
\item[(1)] $h(s,t)$ is Lipschitz on $(-\delta,\delta)\times [a,b]$,
because $\tilde{\gamma}_0$ and $\phi^{-1}$ are of class $C^{1+{\rm Lip}}$ and
$\tilde{\zeta}^{*}$ is of class $C^\infty$,
\item[(2)] $h(0,t)=\phi^{-1} (\tilde{\gamma}_0(t))=\gamma_0(t)$ for all $t\in [a,b]$,
\item[(3)] $t\mapsto h(s,t)$ is of class $C^{1+{\rm Lip}}$
for every $s\in (-\delta,\delta)$, because $\tilde{\gamma}_0$ and $\phi^{-1}$ are of class $C^{1+{\rm Lip}}$ and
$\tilde{\zeta}^{*}$ is of class $C^\infty$,
\item[(4)] the function $\zeta^{*}(t):=\frac{dh}{ds}(0,t)=D \phi^{-1}_{\tilde{\gamma}_0(t)} \tilde{\zeta}^{*}(t)$ is
Lipschitz in $[a,b]$, 
\item[(5)] $h(s,a)=\phi^{-1} (\tilde{\gamma}_0(a))=p$  and 
$h(s,b)=\phi^{-1} (\tilde{\gamma}_0(b))=q$ for all $s\in (-\delta,\delta)$, because $\tilde{\zeta}^{*}$ is proper.
\end{enumerate}
Therefore, $h$ is a proper variation of $\gamma_0$ along $\zeta^{*}$.
Furthermore, 
$$\begin{array}{rcl}
\nrm{\zeta^{*}(t)-\zeta(t)} & = & \nrm{D \phi^{-1}_{\tilde{\gamma}_0(t)}\tilde{\zeta}^{*}(t) - \zeta(t)} \\
{} & \leq & K\,\nrm{\tilde{\zeta}^{*}(t) - \tilde{\zeta}(t) },\\
{} & \leq & K\,\varepsilon \\
\end{array}
$$
where $K=\displaystyle{\max_{x \in\phi(U)}} \nrm{D \phi_x^{-1}}$.
We have $K<+\infty$  because we can shrink the domain $U$ to a compact neighbourhood
of $\gamma_0([a,b])$. Since $K\,\varepsilon$ can be  arbitrarily small the lemma follows.
\end{proof}

\begin{prop}\label{geodesic:char}
The length functional $L:\Omega\to \R$
attains an absolute minimum at a curve $\gamma\in\Omega$
which is also an absolute minimum of the energy functional $E:\Omega\to \R$. Moreover this minimum satisfies
\begin{enumerate}
\item[$(1)$]\, 
$\nrm{\gamma'(t)}$ is constant,
\item[$(2)$]\, $\gamma''(t)\perp T_{\gamma(t)}M$ for almost every
$t\in [a,b]$.
\end{enumerate}
\end{prop}

\begin{proof}
For any $c>0$ define $\Omega^c$ to be the set of curves
$\gamma\in\Omega=\Omega_{p,q}(a,b)$ such that ${\rm Lip}(\gamma')\leq c$.
Fix $c>0$ such that $\Omega^c\neq \emptyset$.
It is readily seen that the set of derivatives
$\{\, \gamma'\,:\, \gamma\in\Omega^ c\,\}$ is equicontinuous.
Hence by the Ascoli-Arzel\'a's theorem the set
$\Omega^ c$ is compact. Since the functional $L$ is continuous on $\Omega$, by Weierstrass Minimum Principle there is a curve
$\gamma_0\in\Omega^c$ such that $L(\gamma_0)$ is the minimum value
of the length on $\Omega^c$. The trace $\Gamma=\gamma_0([a,b])$ is
a $1$-dimensional manifold of class $C^ 1$, for otherwise it wouldn't minimize the length. Hence we can reparametrize $\gamma_0$
to have constant speed. Thus we  assume  $\gamma_0$ satisfies $\nrm{\gamma_0'(t)}=c_0$ for all $t\in [a,b]$, with $c_0\leq c$.
Off course $\gamma_0$ is also the length absolute minimum
in $\Omega$.
By Jensen's inequality,
for any other curve $\gamma\in\Omega$  we have
\begin{align*}
E(\gamma_0) & = c^ 2 (b-a) = \frac{1}{b-a}\, L(\gamma_0)^ 2
\leq  \frac{1}{b-a} \, L(\gamma)^2\\
&= (b-a)\,\left(  \frac{1}{b-a}\,\int_a^b \nrm{\gamma'(t)}\, dt \right)^2\\
& \leq \int_a^b \nrm{\gamma'(t)}^2\, dt = E(\gamma)\;,
\end{align*}
which proves that $\gamma_0$ is also an absolute minimum of
$E:\Omega\to\R$.

Let us now prove item (2).
We can split  $\gamma$ in several pieces so that each part 
is contained in the domain of a single chart.
Clearly each subcurve of $\gamma$ also  minimizes both $L$ and $E$.
Thus, since conclusion (2) of Proposition~\ref{geodesic:char} is local,
we can assume that $\gamma([a,b])$ is contained in the domain of a single chart.
This will allow us to apply Proposition~\ref{var:exist}.
We claim that for any continuous proper vector field 
$\xi$ tangent to $M$ along $\gamma$
$$\displaystyle{\int_a^b \langle \gamma''(t),\xi(t) \rangle \,dt =0}.$$
We first prove this for Lispchitz vector fields associated to proper variations of $\gamma$.
Let $h$ be a proper variation of $\gamma$ along the Lispchitz vector field $\xi$.
Because $E(\overline{h}(s))$ attains its minimum value at $s=0$,
\begin{align*}
0 =   & \frac{d}{ds} E(\overline{h}(s))\vert_{s=0}  = \int_a^b
\frac{d}{ds} \langle \frac{d h}{dt}, \frac{d h}{dt} \rangle \, dt\\
&=  2 \int_a^b
\langle \gamma'(t), \frac{d}{dt}\left[ \frac{dh}{ds}(0,t)\right]  \rangle \, dt\\
& = 2 \int_a^b
 \langle \gamma'(t), \xi'(t) \rangle \, dt \\
 & =  [\langle \gamma'(t),\xi (t) \rangle ]_a^b -  \int_a^b \langle \gamma''(t),\xi(t) \rangle \,dt
 \; .
\end{align*}
Thus, since $\xi$ is proper, $[\langle \gamma'(t),\xi(t) \rangle ]_a^b =0$
and
$$\displaystyle{\int_a^b \langle \gamma''(t),\xi(t) \rangle \,dt =0}.$$
Consider now any continuous  proper vector field $\zeta(t)$ tangent to $M$ along $\gamma$,
and fix a small number $\varepsilon>0$.
By Proposition~\ref{var:exist},
 there is a proper variation $h$ of $\gamma$ along a Lipschitz vector field $\xi$
 which is  $\varepsilon$-close to $\zeta$.
From the previous claim we know that
$\int_a^b \langle \gamma''(t),\xi(t) \rangle \,dt =0$.
Therefore
\begin{align*}
\Mod{ \int_a^b \langle \gamma''(t),\zeta(t) \rangle \,dt }
&\leq \Mod{ \int_a^b \langle \gamma''(t),\zeta(t) \rangle  -  \langle \gamma''(t),\xi(t) \rangle \,dt }\\
&\leq \int_a^b \Mod{ \langle \gamma''(t),\zeta(t) - \xi(t) \rangle } \,dt \\
&\leq \int_a^b \nrm{\gamma''(t)} \nrm{\zeta(t)-\xi(t)}\, dt\\
&\leq c\,(b-a)\,\varepsilon\;, 
\end{align*}
where $c={\rm Lip}(\gamma')$.
Consequently, since $\varepsilon$ can be arbitrarily small,
$$\displaystyle{\int_a^b \langle \gamma''(t),\zeta(t) \rangle \,dt =0}.$$ 
To finish the proof we need to establish (2).
Consider continuous proper vector fields $\zeta_1, \ldots, \zeta_{n-1}$ along $\gamma$, tangent to $M$,
 such that  $\{\zeta_1(t), \cdots, \zeta_{n-1}(t)\}$ is a basis of $T_{\gamma(t)} M$ for all $t \in (a,b)$. 
 Given any continuous function $g:[a,b] \to \R$, since $g \,\zeta_j$ is a 
continuous proper vector field we have 
$\int_a^b \langle \gamma''(t), g(t) \zeta_j(t) \rangle \,dt =0$,
which implies that
$$ \int_a^b \langle \gamma''(t), \zeta_j(t) \rangle\, g(t)\,dt =0\;. $$ 
Thus, because $g$ is arbitrary,
$$\langle \gamma''(t), \zeta_j(t) \rangle=0 \quad \forall\, 1\leq j\leq n-1\;, $$
for almost every $t\in [a,b]$, which implies that
$\gamma''(t)\perp T_{\gamma(t)}M$ for the same values of 
$t\in [a,b]$.
\end{proof}

\bigskip

By the previous proposition given points $p,q\in M$,
there is a curve $\gamma\in\Omega$ connecting $p$ and $q$,
of minimum length $\ell=L(\gamma)$, which can be reparametrized on the interval $[0,\ell]$ to have unit speed, i.e., $\nrm{\gamma'(t)}=1$ for every $t\in [0,\ell]$. As usual,
such a curve will be referred as {\em minimizing unit geodesic}
from $p$ to $q$.

\begin{prop}\label{geo:lip}
If there is a Lipschitz normal field $\eta$ to $M$ with constant norm, $\nrm{\eta(x)}=r$  for every $x\in M$,
and Lipschitz constant ${\rm Lip}(\eta)\leq 1$,
then every minimizing unit geodesic $\gamma:[0,\ell]\to M$
satisfies ${\rm Lip}( \gamma')\leq 1/r$.
\end{prop}

\bigskip

\begin{proof}
Since $\eta$ is a normal field we have
$\langle \gamma'(t),\eta(\gamma(t))\rangle=0$ for every $t$.
Both factors in this product are Lipschitz. By Theorem~\ref{Rademacher} (Rademacher's theorem)
these functions are differentiable
almost everywhere. Applying Leibnitz rule at points
where both these functions are differentiable we get
$ \langle \gamma''(t),\eta(\gamma(t))\rangle +
 \langle \gamma'(t),(\eta\circ\gamma)'(t)\rangle=0$.
 On the other hand, by Proposition~\ref{geodesic:char} we have that
 $\gamma''(t)$ is collinear with $\eta(\gamma(t))$ for almost every
 $t$. But since this orthogonality is  equivalent to the identity
 $\gamma''(t)=\frac{1}{r^2}\,
 \langle \gamma''(t),\eta(\gamma(t))\rangle  \, \eta(\gamma(t))$,
 it follows that
$$ \gamma''(t)= - \frac{1}{r^2}\,
 \langle \, \gamma'(t), \, (\eta\circ\gamma)'(t)\, \rangle  \; \eta(\gamma(t)) $$
 for almost every $t$.
 Now, by assumption $\nrm{\gamma'(t)}=1$ and
 $\nrm{\eta(\gamma(t))}=r$ for all $t$. Since ${\rm Lip}(\eta)\leq 1$,  we must have ${\rm Lip}(\eta\circ \gamma)\leq 1$, and hence
 $\nrm{(\eta\circ \gamma)'(t)}\leq 1$ for  almost every $t$.
Therefore
$$\nrm{\gamma''(t)}\leq \frac{1}{r^2}\,\nrm{\gamma'(t)}\,
\nrm{(\eta\circ \gamma)'(t)}\,\nrm{\eta(\gamma(t))}\leq \frac{1}{r}$$  
for  almost every $t$. Finally, by the mean value theorem
\begin{align*}
 \nrm{\gamma'(t)-\gamma'(t_0)} &= \nrm{\int_{t_0}^{t}
\gamma''(s)\,ds }
 \leq
\int_{t_0}^{t} \nrm{\gamma''(s)}\,ds \\
& \leq \int_{t_0}^{t} \frac{1}{r} \,ds
\leq \frac{1}{r}\,\Mod{t-t_0} \;,
\end{align*}
which proves that ${\rm Lip}( \gamma')\leq 1/r$.
\end{proof}

\section{A Sturm-Liouville Lemma}
\label{sec:6}
In this section we prove a Sturm-Liouville lemma that we will use, in the next section, in the proof of the ``if" part of Theorem~\ref{main}.

\bigskip

\begin{lema}\label{s1} Consider  a curve  $\gamma:[0,+\infty[\to \R^n$ of class $C^{1+{\rm Lip}}$
 such that
\begin{enumerate}
\item[$\mbox{\rm (a)}$] $\nrm {\gamma(0)}=r$,
\item[$\mbox{\rm (b)}$] $\langle \gamma(0),\gamma'(0)\rangle = 0$,
\item[$\mbox{\rm (c)}$] $\nrm{\gamma'(t)}=1$, for all $t\geq 0$,
\item[$\mbox{\rm (d)}$] ${\rm Lip}(\gamma')\leq 1/r$.
\end{enumerate}
Then for every $t\in [0,\pi\,r]$, \, $\nrm{\gamma(t)}\geq r$.
\end{lema}

\begin{proof}
First define, for every $t \in [0,+\infty[$,
$$\varphi(t):=\langle \gamma(t),\gamma(t) \rangle = \nrm{\gamma(t)}^2.$$
We have that
$$\varphi'(t)=2 \langle \gamma(t), \gamma'(t) \rangle.$$
By Theorem~\ref{Rademacher} (Rademacher's theorem)
function $\gamma'$ is differentiable
almost everywhere. Furthermore, we have that $\nrm{\gamma''(t)} \leq 1/r$
for almost every $t$, because ${\rm Lip}(\gamma')\leq 1/r$.
Therefore, for almost every $t$, we can obtain that
$$
\begin{array}{rcl}
\varphi''(t) & = & 2 \,(\, \nrm{\gamma'(t)}^2 + \langle \gamma(t),\gamma''(t) \rangle) \\
{} & \geq & 2 \, (1 - \nrm{\gamma(t)}\nrm{\gamma''(t)}) \\
{} & \geq & 2 \, \left(1-\frac{\nrm{\gamma(t)}}{r}\right) \\
{} & = & 2 \, \left(1-\sqrt{\frac{\varphi(t)}{r^2}}\right) \\
{} & \geq & 2\,\left( 1-\left(1+\frac{1}{2}\left(\frac{\varphi(t)}{r^2}-1 \right) \right) \right) \\
{} & = & 1-\frac{\varphi(t)}{r^2}.
\end{array}
$$
Now define
$$\psi(t):= 1- \frac{\varphi(t)}{r^2}.$$
Derivating $\psi$ and using the above inequality we get that, for almost every $t$,
$$\psi''(t) = -\frac{1}{r^2} \varphi''(t) \leq -\frac{1}{r^2} \psi(t)$$
which is equivalent to
$$\psi''(t)+\frac{1}{r^2}\psi(t) \leq 0.$$
Moreover, it is clear that,
$$\psi(t) < 0 \Leftrightarrow \varphi(t) > r^2,$$
or equivalently,
$$\psi(t) < 0 \Leftrightarrow \nrm{\gamma(t)} > r.$$
Let $b>0$ be the first time $t>0$ such that $\nrm{\gamma(t)}=r$, with $b=+\infty$ if no such time
exists. Our goal is to show that $b\geq \pi\, r$,
something obvious if $b=+\infty$. 
Hence we assume that $b$ is finite.
By definition we have $\psi(t)<0$ for all $t\in (0,b)$.
Define on this interval
$$q(t) := \frac{\psi''(t)}{\psi(t)}+\displaystyle\frac{1}{r^2} \;. $$
Assumptions (c) and (d) imply that $\varphi''$, and hence $\psi''$,
are measurable bounded functions. Therefore $q(t)$ is locally integrable
on $(0,b)$. Clearly  $q(t)\geq 0$ for all $t\in (0,b)$ and
$$\psi''(t) + \left(\frac{1}{r^2} - q(t) \right) \psi(t)=0.$$
Since $\varphi(0)=\varphi(b)=r^2$ there is a first time $0<t_0<b$ such that
$\langle \gamma(t_0),\gamma'(t_0) \rangle =  \frac{1}{2} \varphi'(t_0)=0$. 
Clearly, $\psi(t_0)<0$ and $\psi'(t_0)=0$. 
By Sturm-Liouville theory, see~ \cite{DH},  we can compare the solution of the previous equation
with the solution of the problem
$$\left\{
\begin{array}{l}
\tilde{\psi}''(t) + \frac{1}{r^2} \tilde{\psi}(t)=0, \\
\tilde{\psi}(t_0)=\psi(t_0) \\
\tilde{\psi}'(t_0) = \psi'(t_0) \\
\end{array}
\right.
$$
which is given by $\tilde{\psi}(t)=\psi(t_0)\,\cos\left(\frac{t-t_0}{r}\right)$, to conclude that
$$\psi(t) \leq \tilde{\psi}(t),$$
while $\tilde{\psi}(t) <0$. Therefore, the first zero $b$ of $\psi$ satisfies
$$b  \geq t_0 + r \frac{\pi}{2}.$$
Now let, 
$$\psi_1(t):=\psi(-t).$$
Clearly, $\psi_1'(t)=-\psi'(-t)$ and $\psi_1''(t)=\psi''(-t)$.
Therefore, we can replace in the above differential equations, $\psi$ by $\psi_1$. 
Consequently, we can repeat the same arguments as before for $t \leq t_0$ to obtain that
$$0  \leq t_0 - r \frac{\pi}{2} .$$
Hence $b\geq \pi\, r$,
which completes the proof.
\end{proof}

\medskip

\section{Deriving Regularity }
\label{sec:7}
In this section we shall prove the `if' part of
 Theorem~\ref{main}, i.e., we prove the following proposition.

\begin{prop}\label{mainglobal}
Given an open set $U\subseteq\R^n$ assume  that
\begin{enumerate}
\item[$(1)$]\, there is $\eta:\partial U\to\R^n$
such that
\begin{enumerate}
\item[$\mbox{\rm (i)}$]\, $\eta$ is a normal vector field  along $\partial U$,
\item[$\mbox{\rm (ii)}$]\, $\nrm{\eta(x)}=r$, for every $x\in\partial U$,
\item[$\mbox{\rm (iii)}$]\, ${\rm Lip}(\eta)\leq 1$,
\end{enumerate}
\item[$(2)$]\, $d_{\partial U}(x,y) \geq \frac{\pi}{2}\, \nrm{x-y}$ implies $\nrm{x-y} \geq 2\,r$, for all $x,y \in \partial U$.
\end{enumerate}
Then $U$ is $r$-regular.
\end{prop}

\begin{proof}
By Proposition~\ref{localnormal}, $U$ is locally $r$-regular. Therefore, we are left to prove that 
the level set $\partial U=f^{-1}(0)$ is globally bounded between two  balls of radius $r$
tangent at $x$, for any $x \in \partial U$.
Let $\eta:\partial U\to\R^n$ be such that conditions $(1)$ $\mbox{\rm (i)}$, $\mbox{\rm (ii)}$ and $\mbox{\rm (iii)}$ are satisfied.
Suppose that $\eta$ points outward $U$.
We want to show that $B\left(x + \eta(x),r\right) \subseteq \overline{U}^c$ and $B\left(x - \eta(x),r\right) \subseteq U$.
We shall prove the first inclusion, since the proof of the second one is analogous. 
Suppose, by contradiction, that there exists $y \in \partial U$ such that $y \in  B\left(x + \eta(x),r\right)$.
Let $\gamma:[0,\ell] \to \R^n$ be a minimizing unit geodesic from $x$ to $y$.
By Proposition~\ref{geo:lip}, ${\rm Lip}( \gamma')\leq 1/r$.
Therefore, it follows from Lemma~\ref{s1} that
$$d_{\partial U}(x,y) = \ell \geq \pi r \geq \frac{\pi}{2}\nrm{x-y}.$$
Consequently, condition $(2)$ implies that $\nrm{x-y} \geq 2r$,
which is a contradiction.

\end{proof}

\section{Global Characterization of Regularity }
\label{sec:8}

To finish the proof of Theorem~\ref{main} we are left to prove the following proposition.

\begin{prop}\label{intrinsic}
Let $U\subseteq\R^n$ be an $r$-regular set.
Then we have that
$$d_{\partial U}(x,y) \geq \frac{\pi}{2}\, \nrm{x-y} \hspace{0.2cm} \mbox{\rm implies} \hspace{0.2cm}  \nrm{x-y} \geq 2\,r,$$
for all $x,y \in \partial U.$
\end{prop}

\begin{lema}
Let $U\subseteq\R^n$ be an $r$-regular set.
The projection $\pi:N_r\to\partial U$ that to each point $x\in N_r$ associates
the nearest point $\pi(x)\in\partial U$ 
is continuous.
\end{lema}

\begin{proof}
 The minimizing projection $\pi$ is well defined by~\cite[Lemma 5]{DT2}. Consider any sequence $x_n \in N_r$ such that $x_n \rightarrow x$ with $x \in N_r$ and suppose, by contradiction, that $\pi(x_n) \not\rightarrow  \pi(x)$. Therefore, there exists some subsequence $x_{n_{k}}$ such that
$\pi(x_{n_{k}}) \rightarrow  y$ with $y\in\partial U$ and $y \neq \pi(x)$. But then
$$\displaystyle{\nrm{x-\pi(x)} = \lim_n \nrm{x_{n_{k}} - \pi(x_{n_{k}})}=\nrm{x-y}},$$
which implies that $\pi(x)=y$.

\end{proof}

\begin{lema} Let $U\subseteq\R^n$ be an $r$-regular set.
Given $x,y\in\partial U$ such that $\nrm{x-y}<2r$, if $B$ denotes the closed ball
with diameter $[x,y]$ then the subspace
$\partial U  \cap B$ is connected.
\end{lema}

\begin{proof}
Take $x,y\in \partial U$ such that $\nrm{x-y}<2r$.
We claim that $[x,y]\subseteq N_r$.
Given $z\in [x,y]$, assume $\nrm{x-z}\leq \nrm{y-z}$.
Then
$$ d(z,\partial U)\leq \nrm{x-z} \leq \frac{1}{2}\nrm{x-y} < r\;, $$
which proves that $z\in N_r$. Otherwise, if $\nrm{x-z}\geq \nrm{y-z}$
then
$$ d(z,\partial U)\leq  \nrm{y-z} \leq \frac{1}{2}\nrm{x-y} < r\;, $$
which again proves that $z\in N_r$. Hence $[x,y]\subseteq N_r$,
and $\pi([x,y])$ is a connected curve joining $x$ and $y$ in $\partial U$.
To finish we just need to show that $\pi([x,y])\subseteq B$. Let $p=\frac{x+y}{2}$ be the
centre of $B$ and notice that $B$ has radius $r=\nrm{x-p}=\nrm{y-p}$.
Given $z\in [x,y]$, assume $\nrm{x-z}\leq \nrm{y-z}$.
Then
$$
\begin{array}{rcl}
\nrm{\pi(z)- p}  & \leq & \nrm{\pi(z)-z}+\nrm{z-p}  \\
{} & \leq & \nrm{x-z}+\nrm{z-p} \\
{} & = & \nrm{x-p}\;,
\end{array}
$$
which proves that $\pi(z)\in B$. 
Otherwise, if $\nrm{x-z}\geq \nrm{y-z}$
then
$$ 
\begin{array}{rcl}
\nrm{\pi(z)-p}  & \leq & \nrm{\pi(z)-z}+\nrm{z-p}  \\
{} & \leq & \nrm{y-z}+\nrm{z-p} \\
{} & = & \nrm{y-p}\;,
\end{array}
$$
which again proves that $\pi(z)\in B$. 

\end{proof}

\bigskip

\begin{coro}\label{partialU}
 Let $U\subseteq\R^n$ be an $r$-regular set.
Given $ x,y \in \partial U$ such that   $0<\nrm{x-y}<2r$
there is $z\in \partial U$  satisfying 
\begin{enumerate}
\item[$\mbox{\rm (a)}$]\,
$\langle z-\frac{x+y}{2},  y-x \rangle =0$,
\item[$\mbox{\rm (b)}$]\,   $\nrm{z-\frac{x+y}{2}}\leq \frac{1}{2}\nrm{x-y}$,
\item[$\mbox{\rm (c)}$]\,  the line through $z$ normal to $\partial U$ is contained in the plane determined
by the points $x$, $y$ and $z$.
\end{enumerate}
\end{coro}

\begin{proof}
Let $\delta=\frac{1}{2}\nrm{x-y}$ and consider the continuous function
$f:\partial U\cap \overline{B}_{\delta}(\frac{x+y}{2})\to\R$ defined by 
 $f(z)=\langle z-\frac{x+y}{2},  y-x \rangle$.
Remark that $f(x)<0$ and $f(y)>0$.
Thus, since $\partial U\cap \overline{B}_\delta(\frac{x+y}{2})$
is connected, $f(z)=0$  for some
$z\in \partial U\cap \overline{B}_\delta(\frac{x+y}{2})$.
Let $H$ denote the hyperplane $\{\, z\in\R^n\,\colon\, f(z)=0\,\}$.
We know there is at least one point $z\in\partial U\cap H$ satisfying $\nrm{z-\frac{x+y}{2}}\leq \delta$.
Hence we can pick such a point $z$ minimizing the distance $\nrm{z-\frac{x+y}{2}}$.
It follows that the vector $z-\frac{x+y}{2}$ is orthogonal to $\partial U\cap H$ at $z$.
Since the normal to $\partial U\cap H$ at $z$ is the orthogonal projection onto $H$ of the normal
 to $\partial U$ at $z$, this normal direction to $\partial U$ at $z$ lies in the plane
 spanned by $y-x$ and  $z-\frac{x+y}{2}$. This implies that the  normal line to $\partial U$ through $z$ is contained in the plane determined by the points $x$, $y$ and $z$.
 
\end{proof}

\begin{figure}[h]
\begin{center}
\includegraphics[width=0.5\textwidth]{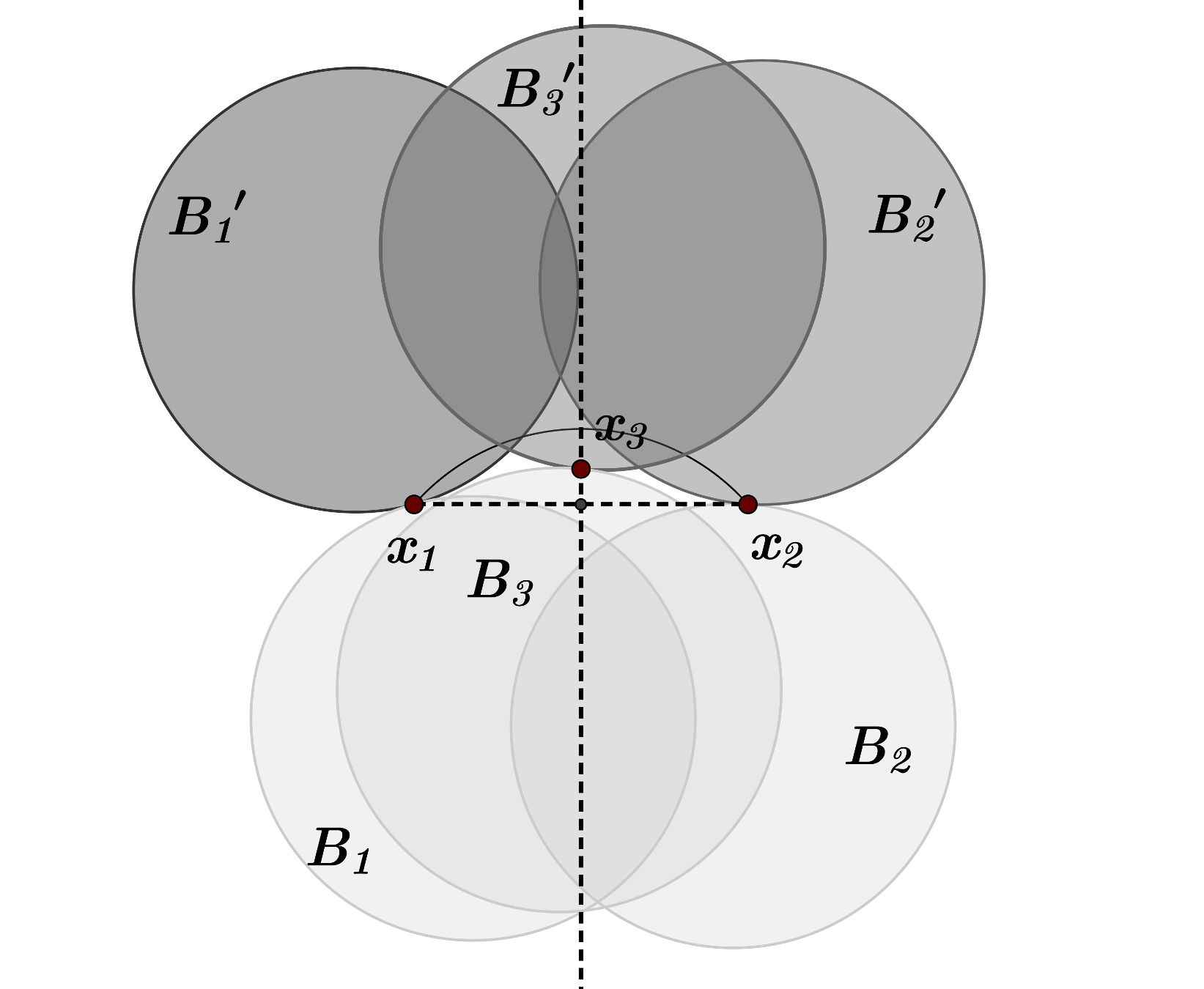}
\end{center}
\caption{A configuration of balls satisfying the Six Ball Lemma}
\label{fig:3}
\end{figure}

\begin{lema}[Six Ball Lemma]
Given three pairs of $r$-balls, $(B_i, B_i')$, $i=1,2,3$,
such that (see Fig.~\ref{fig:3}):
\begin{enumerate}
\item[$\mbox{\rm (a)}$]\, $B_i$ and $B_i'$ are tangent at $x_i$,  $i=1,2,3$,
\item[$\mbox{\rm (b)}$]\, $(B_1\cup B_2 \cup B_3) \cap  (B_1'\cup B_2' \cup B_3') = \emptyset$,
\item[$\mbox{\rm (c)}$]\,  $0<\nrm{x_1-x_2}<2r$,
\item[$\mbox{\rm (d)}$]\,  $\nrm{x_3-x_1}=\nrm{x_3-x_2}$ and $\nrm{ x_3-\frac{x_1+x_2}{2} }\leq \frac{1}{2}\nrm{x_1-x_2}$,
\item[$\mbox{\rm (e)}$]\, The line through the centres of $B_3$ and $B_3'$ is contained in the plane
determined by the points $x_1$, $x_2$ and $x_3$.
\end{enumerate}
Then we have
$$\nrm{ x_3-\frac{x_1+x_2}{2} }\leq r - \sqrt{ r^2 - 
\frac{\nrm{x_1-x_2}^2}{4} }.$$
\end{lema}

\begin{proof}
The proof goes by contradiction.
Let  $x_0 :=\frac{x_1+x_2}{2}$ and $s:=\nrm{x_1-x_2}$, so that $\nrm{x_3-x_0}\leq\frac{s}{2}$,
but assume that $h:=\nrm{x_3-x_0}>r-\sqrt{r^2-\frac{s^2}{4}}$.
We shall derive a contradiction from this.
Let $\eta= \frac{r}{\nrm{x_3-x_0}}(x_0-x_3)$
and $\eta'=\frac{r}{\nrm{x_2-x_1}}(x_2-x_1)$.
These are orthogonal vectors with $\nrm{\eta}=\nrm{\eta'}=r$.
Consider also the unique vector $\eta_3\in\R^ n$ such that (i) $\nrm{\eta_3}=r$,
(ii) $x_3+\eta_3$ is the centre of $B_3$, and (iii) $x_3-\eta_3$ is the centre of $B_3'$.
By assumption (e), the vector $\eta_3$ lies in the plane spanned by $\eta$ and $\eta'$.
We can assume that $\langle \eta,\eta_3\rangle\geq 0$. Otherwise we work with
$-\eta_3$ and $B_3'$ instead  of $\eta_3$ and $B_3$. Hence we can write
$\eta_3=(\cos\theta)\eta + (\sin\theta)\eta'$ with $\theta\in[-\frac{\pi}{2},\frac{\pi}{2}]$.
Exchanging the roles of $x_1$ and $x_2$ (if necessary) we can also assume that $\theta\in[0,\frac{\pi}{2}]$.
Now, consider the function $\varphi:[0,\frac{\pi}{2}]\to \R$,
\begin{align*}
\varphi(\theta) &= \nrm{x_3+\eta_3-x_2}^ 2 =
\nrm{x_3-x_2 + (\cos\theta)\eta + (\sin\theta)\eta'}^ 2\\
&= \nrm{ x_0-x_2 -\frac{h}{r}\eta  + (\cos\theta)\eta + (\sin\theta)\eta'}^ 2\\
&= \nrm{  \left(  \cos\theta -\frac{h}{r} \right)\eta + \left( \sin\theta -\frac{s}{2r}\right)\eta'}^ 2\\
&= \left(  r \cos\theta - h \right)^ 2 + \left( r \sin\theta -\frac{s}{2}\right)^ 2\;.
\end{align*}
Note that since $h> r-\sqrt{r^2-\frac{s^2}{4}}$,
$$\varphi(0)=(r-h)^ 2+(s/2)^2 < \left( r^2-\frac{s^2}{4} \right) + \frac{s^2}{4}=r^ 2\;. $$
Also, since $h\leq s/2$ and $s<2r$,
$$\varphi(\frac{\pi}{2})= h^2+(r-s/2)^ 2 \leq r^ 2 + s\left(\frac{s}{2}-r\right)<r^ 2\;. $$
It is not difficult to see that the function $\varphi(\theta)$ attains its minimum value
at the argument of the vector $ h\,\eta + \frac{s}{2} \eta'$ w.r.t. the orthonormal frame
$\{\eta,\eta'\}$, and its maximum value at one of the boundary points $\theta=0$ or
$\theta=\frac{\pi}{2}$. Thus, from the previous inequalities it follows that
$\nrm{x_3+\eta_3-x_2}^ 2=\varphi(\theta)<r^ 2$.
This shows that $x_2\in B_3$, and hence $B_3$ intersects both $B_2$ and $B_2'$,
thus contradicting assumption (b).

\end{proof}

\bigskip

\begin{rmk} For any $0< s < 2r$,
$$r-\sqrt{r^ 2-\frac{s^ 2}{4}} < \frac{s}{2} $$
with equality at the endpoints of $]0,2r[$.
Hence the conclusion of the Six Ball Lemma is an improvement of the
{\em a priori} bound in assumption $\mbox{\rm (d)}$.
\end{rmk}

\bigskip

Define $\varphi:[0,2r]\to [0,r]$,
$$\varphi(s)= r - \sqrt{ r^2 - \frac{s^2}{4}  } $$
and $\psi:[0,2r]\to \R$,
$$\psi(s)=  \sqrt{ \frac{s^2}{4} +\varphi(s)^2  }
= \sqrt{2 r^2-2 r\,\sqrt{r^2-\frac{s^2}{4}}}\;$$
(see Fig.~\ref{fig:4}).

\begin{figure}[h]
\begin{center}
 \includegraphics*[scale={.7}]{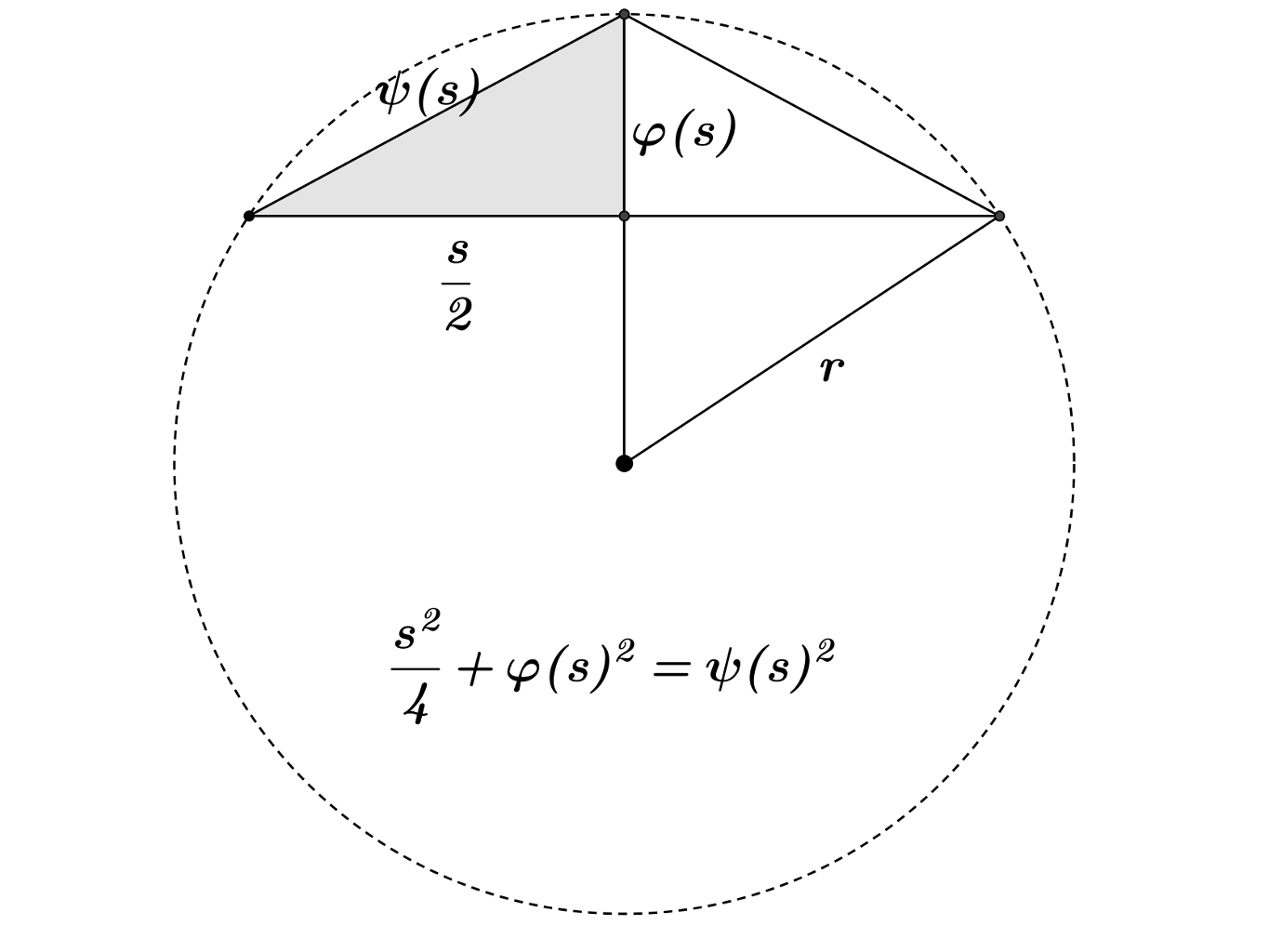} 
\end{center}
\caption{Geometric meaning of $\psi(s)$ and $\varphi(s)$}
\label{fig:4}
\end{figure}

\begin{coro}\label{mean:bdy:pt}
 Let $U\subseteq\R^n$ be an $r$-regular set.
Given $ x,y \in \partial U$ such that  $\nrm{x-y}<2r$
there is $z\in\partial U$  satisfying 
\begin{enumerate}
\item[$(1)$]\,$\langle z-\frac{x+y}{2},  y-x \rangle =0$,
\item[$(2)$]\,  $\nrm{z-\frac{x+y}{2}}\leq \varphi( \nrm{x-y})$,
\item[$(3)$]\, $\nrm{z-x}\leq\psi(\nrm{x-y})$,
\item[$(4)$]\, $\nrm{z-y}\leq\psi(\nrm{x-y})$.
\end{enumerate}
\end{coro}

\begin{proof} Using Corollary~ \ref{partialU}, an application of the Six Ball Lemma gives $(1)$ and $(2)$.
 Pithagoras' Theorem implies $(3)$ and $(4)$.
\end{proof}

\bigskip

\begin{lema} \label{Psi:lemma}
Given $s\in [0,2r]$, 
 $$ \lim_{n\to+\infty}  {2^n}\,\psi^n(s)= 
 2 r \, \arctan\left( \frac{s }{\sqrt{4 r^2 - 
s^2} } \right) \;. $$
\end{lema}

\begin{proof}[Geometric Proof]
 Consider a chord $\overline{AB}$ of length $s<2r$ in a circle $\mathscr{C}$
 of radius $r$. Let $\Gamma$ denote the
 shortest arch of $\mathscr{C}$ connecting
 the points $A$ and $B$, which has length
  $2 r \, \arctan\left( \frac{s }{\sqrt{4 r^2 - 
s^2} } \right)$.

 We can approximate $\Gamma$ by a
 polygonal line $\Gamma_n$ consisting of  $2^n$ equal sides 
 and all vertexes in $\mathscr{C}$.
 Recursively, we set $\Gamma_0=\overline{AB}$ and define $\Gamma_{n}$
 to be the polygonal line obtained from $\Gamma_{n-1}$
 replacing each side  $\overline{XY}$ of $\Gamma_{n-1}$
 by the two equal sides $\overline{XZ}$ and $\overline{ZY}$,
 where $Z$ is the nearest  point where the line that bisects $\overline{XY}$ intersects the circle $\mathscr{C}$.
It is easy to check, inductively in $n$, that
 ${\rm length}(\Gamma_n)= 2^n\, \psi^n(s)$ for every $n\geq 0$.
 Thus
  $$ 
\begin{array}{rcl}
 \displaystyle{ \lim_{n\to+\infty}}  {2^n}\,\psi^n(s) & = &
  \displaystyle{\lim_{n\to+\infty}} {\rm length}(\Gamma_n) = 
  {\rm length}(\Gamma) \\
  {} & = & 2 r \, \arctan\left( \frac{s }{\sqrt{4 r^2 - 
s^2} } \right) \;. 
\end{array}
$$
\end{proof}

\bigskip

The following theorem proves Proposition~\ref{intrinsic}.

\begin{teor}
 Let $U\subseteq\R^n$ be an $r$-regular set.
Given $ x,y \in \partial U$ such that  $\nrm{x-y}<2r$,
the intrinsic distance between $x$ and $y$ in $\partial U$
satisfies
$$ d_{\partial U}(x,y) \leq 2 r \arctan\left( \frac{\nrm{x-y} }{\sqrt{4 r^2 - 
\nrm{x-y}^2} } \right)<\pi\,r\;.$$
\end{teor}

\begin{proof}
 Applying Corollary~\ref{mean:bdy:pt} inductively
 we can construct a sequence of polygonal lines
 $\Gamma_n$ with $2^n$ equal sides, and all vertexes
 in $\partial U$. By construction the polygonal
 line $\Gamma_n$ has length $\leq 2^n\,\psi^n(\nrm{x-y})$,
 and can be parametrized as a Lipschitz curve with
 Lipschitz constant $1/r$. The sequence $\Gamma_n$
 is a Cauchy sequence w.r.t. $\nrm{\cdot}_\infty$.
 Therefore in view of Lemma~\ref{Psi:lemma} $\Gamma=\lim_{n\to\infty}\Gamma_n$ is 
 a Lipschitz curve connecting $x$ to $y$ with length
 $${\rm length}(\Gamma)\leq 2 r \arctan\left( \frac{\nrm{x-y} }{\sqrt{4 r^2 - 
\nrm{x-y}^2} } \right) \;.$$
\end{proof}

\bigskip

Together, Proposition~\ref{localmain}, Proposition~\ref{mainglobal}  and Proposition~\ref{intrinsic} prove Theorem~\ref{main}.

\section{Conclusions and Future Work}
\label{sec:9}
In this paper we have provided a characterization of $r$-regular sets in terms of the Lipschitz regularity of normal vector fields to the boundary.

We would like to consider a possible extension of the notion of $r$-regular sets regarding its application to the
problem of surface reconstruction. We also plan to pursue our research on smooth non-deterministic
dynamical systems where the `smoothness' of the boundary of $r$-regular sets will be applied.

\bigskip

Both authors would like to thank Armando P. Machado and Alessandro Margheri for valuable
suggestions on the problems related to the Sturm-Liouville section.
The first author was supported by ``Funda\c{c}\~{a}o para a Ci\^{e}ncia e a Tecnologia''
 through the Program POCI 2010 and the Project
 ``Randomness in Deterministic Dynamical Systems and Applications'' (PTDC-MAT-105448-2008).
The second author was partially financed by FEDER Funds through ``Programa Operacional Factores de Competitividade - COMPETE'' and by Portuguese Funds through FCT - ``Funda\c{c}\~{a}o para a Ci\^{e}ncia e a Tecnologia'', within the Project PEst-OE/MAT/UI0013/2014.

\end{document}